%% file: main.tex
\documentclass[12pt]{article}

\input{preamble}

\begin{document}
\maketitle

\begin{abstract}
A fundamental problem in pattern avoidance is describing the asymptotic behavior of the extremal function and its generalizations. 
We prove an equivalence between the asymptotics of the graph extremal function for a class of bipartite graphs and the asymptotics of the matrix extremal function.
We use the equivalence to prove several new bounds on the extremal functions of graphs.
We develop a new method to bound the extremal function of hypergraphs in terms of the extremal function of their associated multidimensional matrices, improving the bound of the extremal function of $d$-permutation hypergraphs of length $k$ from $O(n^{d-1})$ to $2^{O(k)}n^{d-1}$.
\end{abstract}

\setlength\abovedisplayskip{2pt}
\setlength\belowdisplayskip{0pt}

\input{background_definitions}
\input{new_results_outline}

\input{graph_extensions}
\input{hypergraph_extensions}
\input{partition_bounds}
\input{lowerbounds}
\input{conclusion}
\input{acknowledgements}

\bibliography{references} 
\bibliographystyle{ieeetr}

\end{document}

%% file: preamble.tex
\usepackage{amsmath} 
\usepackage{amsthm}
\usepackage{amssymb}
\usepackage{hyperref}
\usepackage{setspace}
\usepackage[margin=1in]{geometry}
\usepackage{algorithm}
\usepackage[noend]{algpseudocode}
\usepackage{verbatim}
\usepackage{times}
\usepackage{fancyhdr}

\newtheorem{theorem}{Theorem}[section]
\newtheorem{lemma}[theorem]{Lemma}

\newtheorem{corollary}[theorem]{Corollary}

\theoremstyle{definition}
\newtheorem*{definition}{Definition} 

\theoremstyle{remark}

\renewcommand{\pm}{\text{pm}}
\renewcommand{\th}{^{\text{th}}}

\newcommand{\ex}[2]{\text{ex}\!\left(#1, #2\right)}
\newcommand{\gex}[2]{\text{gex}\!\left(#1, #2\right)}
\newcommand{\hyperex}[3]{\text{ex}_{#1}\!\left(#2, #3\right)}

\title{\vspace{-3ex}Improved bounds on the extremal function of hypergraphs}

\author{William Zhang}

\date{}

%% file: background_definitions.tex
\section{Introduction}

\subsection{Background}

Problems in the extremal theory of pattern avoidance often ask for the number of elements a structure can have without containing a specific substructure \cite{GenesonThesis}.
A fundamental problem in pattern avoidance is describing the asymptotic behavior of the extremal function and its generalizations \cite{GunbyPalv}.
According to Kitaev \cite{kitaev2011patterns}, the introduction to the area of permutation patterns is traditionally attributed to Knuth \cite{knuth1998art}.
One of the largest problems in the study of pattern avoidance is the Stanley-Wilf conjecture, formulated by Stanley and Wilf independently in the late 1980s \cite{bona1997}, which states that the growth rate of the number of permutations avoiding a given permutation pattern is exponential.
The F\"{u}redi-Hajnal conjecture \cite{furedi1992davenport} states that the growth rate of the extremal function of permutation matrices is linear.
Klazar \cite{klazar2000furedi} showed an equivalence between the Stanley-Wilf conjecture and the F\"{u}redi-Hajnal conjecture.
The constants in $2^{O(n)}$ and $O(n)$ from the Stanley-Wilf and F\"{u}redi-Hajnal conjectures are called the \textit{Stanley-Wilf limit} and the \textit{F\"{u}redi-Hajnal limit}, respectively.
Cibulka \cite{cibulka2009constants} later proved a polynomial relationship between the Stanley-Wilf limit and the F\"{u}redi-Hajnal limit of a permutation.
Marcus and Tardos \cite{MarcTard} proved the F\"{u}redi-Hajnal conjecture, and thus the Stanley-Wilf conjecture. Marcus and Tardos's work have since then been generalized in various directions, and many others have significantly sharpened their bounds \cite{fox2013stanley, KlazarMarcus, CibKyncl, TianGeneson, MarcThesis, WeidertThesis, GunbyPalv}.
Following the Marcus-Tardos Theorem, Fox \cite{fox2013stanley} proved exponential upper and lower bounds on the Stanley-Wilf limit, disproving a widely believed conjecture that the Stanley-Wild limit was quadratic in the length of its permutation \cite{steingrimsson2013some}.

This problem has been extensively studied both out of mathematical interest and due to its applications in computational geometry and engineering. 
Mitchell's algorithm \cite{Mitchell} computes a  shortest
rectilinear path avoiding rectilinear obstacles in the plane.
Mitchell showed that the complexity of the algorithm can be bounded by the extremal function of a specific matrix.
Bienstock and Gy\"{o}ri \cite{Bienstock} bounded the complexity of the algorithm by finding sharp upper bounds on the extremal function of that matrix.
Mitchell's algorithm has direct applications in both motion planning in robotics and wire routing in VLSI circuit design \cite{yang1992rectilinear}. 
Furthermore, 
F\"{u}redi \cite{furedi1990} used the extremal function to find an upper bound on the Erd\H{o}s-Moser problem \cite{Erdos} of determining the maximum number of unit distances in a convex polygon.
Aggarwal \cite{aggarwal2015unit} sharpened F\"{u}redi's result on the upper bound on the maximum number of unit distances in a convex polygon.
Some problems from pattern avoidance also emerged in bounding the number of possible lower envelope sequences formed by continuous functions \cite{sharir1995davenport}.

\subsection{Definitions}

%-------------------------------------------------------------
%                               Definitions
%-------------------------------------------------------------

We denote the list $\{1, \ldots, n\} = [n]$.

For some integer $d \geq 2$, a $d$-dimensional matrix is a \textit{block} of numbers on $[n_1]\times \cdots \times [n_d]$.
We denote such a  matrix $A =  (A_{(i_1, \ldots, i_d)})$ where $1\leq i_l \leq n_l$ for each $l\in [d]$.
In this paper we only consider binary matrices, so every entry is a \textit{0-entry} or a \textit{1-entry}.
When  we refer to a $d$-dimensional matrix $A$ having a side length $k$,  we mean that $A$ is a block of numbers on $[k]^d$.  We also refer to a $d$-dimensional matrix as a \textit{$d$-matrix.} We denote the number of $1$-entries in a $d$-matrix $A$ by $w(A)$.

An \textit{$l$-cross section} of matrix $A$ is the set of all the entries 
$A_{(i_1, \ldots, i_d)}$ whose $l\th$ coordinates have the same value.
An \textit{$l$-row} of matrix $A$ is the collection of all the entries $A_{(i_1, \ldots, i_d)}$ whose coordinates other than the $l\th$ coordinate have fixed values.

\begin{sloppypar}
\begin{definition}
Let $\pi_1, \ldots, \pi_{d-1}$ be permutations on $[k]$. Then the matrix $A$ that is defined by
$A_{(i, \pi_1(i), \ldots, \pi_{d-1}(i))} = 1$ for each $i \in [d-1]$ and 0-entries everywhere else is the \textit{$d$-permutation matrix} of length $k$ \textit{constructed} from $\pi_1, \ldots, \pi_{d-1}$. A \textit{permutation matrix} refers to a $2$-permutation matrix.
\end{definition}
\begin{definition}
If $A$ and $B$ have the same dimensions and $B$ can be obtained from $A$ by changing 1-entries in $A$ into 0's, then $A$ \textit{represents} B.
\end{definition}
\end{sloppypar}

If some submatrix of $A$ represents $B$, then $A$ \textit{contains} $B$. Otherwise, $A$ avoids $B$. The extremal function $\ex{B}{n}$ denotes the maximum possible number of 1-entries in an $n\times n$ binary matrix that avoids $B$. We call $B$ \textit{linear} if $\ex{B}{n} = \Theta(n)$.
For $d$-dimensional matrices, $f(B, d, n)$ denotes the maximum possible number of 1-entries in a $d$-matrix of side length $n$ that avoids $B$.

An \textit{ordered hypergraph} is an ordered pair $H = (V, E)$ where $V$ is a linearly ordered set and $E$ is a set of subsets of $V$. Each $v\in V$ is a \textit{vertex} of $H$, and each $e\in E$ is an \textit{edge} of $H$. The \textit{weight} of a hypergraph $H = (V, E)$ is $\sum_{e\in E}|e|$.
For $d\in \mathbb{N}$, a hypergraph $H = (V, E)$ is $d$-uniform if for each $e\in E$,we have $|e| = d$. An \textit{ ordered graph} is a 2-uniform ordered hypergraph.
Because this paper does not deal with unordered graphs and unordered hypergraphs,we refer to ordered graphs and ordered hypergraphs as just graphs and hypergraphs.

\begin{definition}
A hypergraph $A = (V_1, E_1)$ \textit{contains} another hypergraph $B = (V_2, E_2)$ if there exists an increasing injection $f: V_2 \mapsto V_1$ and an injection $g: E_2 \mapsto E_1$ such that for each $e \in E_2$,we have that $f(e)\subset g(e)$. Otherwise, $A$ \textit{avoids} $B$. If $f$ and $g$ are bijections such that $f(e) = g(e)$ for each $e \in E_2$, then $A$ and $B$ are \textit{order-isomorphic}.
\end{definition}

If $G$ is a graph, the extremal function for graphs $\gex{G}{n}$ denotes the maximum possible number of edges in a graph with $n$ vertices such that $A$ avoids $G$. Analogously,we associate two extremal functions for hypergraphs. If $H$ is a hypergraph, then $\hyperex{e}{H}{n}$ denotes the maximum possible number of edges of a hypergraph on $[n]$ that avoids $H$, and $\hyperex{i}{H}{n}$ denotes the maximum possible weight of a hypergraph on $[n]$ that avoids $H$.

\begin{definition}
Let $H = (V, E)$ be a hypergraph with $V = \{v_1,\ldots, v_n\}$ such that $v_1 < \cdots < v_k$. Ifwe can partition $V$ into $d$ sets $I_i = \{v_{k_{i - 1} + 1}, \ldots, v_{k_i}\}$ for $i\in [d]$ with $k_0 = 0$ and $k_d = n$
such that each $e\in E$ contains at most one vertex from each $I_i$, then $H$ is \textit{$d$-partite}. Each $I_i$ is a \textit{part} of $H$.
\end{definition}

\begin{definition}
Given a $d$-dimensional matrix $M$ with dimensions $[n_1]\times \cdots \times [n_d]$, the hypergraph \textit{associated} with $M$ is $H = ([\sum_{i = 1}^{d} n_i], E)$ where for each 1-entry $m_{k_1, \ldots, k_d}$, $E$ contains the edges $\{(\sum_{i = 1}^{j - 1}m_i) + k_j\}_{j = 1}^{d}$. Conversely, for a $d$-partite, $d$-uniform hypergraph $H'$, the $d$-matrix associated with $H'$ is the $d$-matrix $M'$ such that $H'$ is associated with $M'$.
\end{definition}

We see that if $M$ has side length $n$, then $H$ is a $d$-partite graph on $nd$ vertices with each part of size $n$.
A \textit{$d$-permutation hypergraph} of length $k\in \mathbb{Z}^+$ is a  $d$-uniform, $d$-partite hypergraph $H = ([kd], E)$ with parts of size $k$ such that each vertex $v\in [kd]$ is in exactly one edge.
Similarly, a \textit{permutation graph of length $k$} is the graph associated with a $2$-permutation matrix.
We see that every $d$-permutation hypergraph is the hypergraph associated with a $d$-permutation matrix, and vice versa.
Klazar and Marcus \cite{KlazarMarcus} observed that if $G$ and $G$ are $d$-partite, $d$-uniform hypergraphs with $nd$ vertices and parts of size $n$, then $G$ contains $H$ if and only if the the matrix associated with $G$ contains the matrix associated $H$.

\begin{definition}
In a $d$-matrix $P$, the \textit{distance vector} between entries $P_{(a_1, \ldots, a_d)}$ and $P_{(b_1, \ldots, b_d)}$ is $(b_1 - a_1, \ldots, b_d - a_d) \in \mathbb{Z}^d$.
A vector $\textbf{x}\in \mathbb{Z}^d$ is \textit{$r$-repeated} in a permutation matrix $P$ if $\textbf{x}$ occurs as the distance vector of at least $r$ pairs of 1-entries.
\end{definition}
%------------------------------This stuff is commented out --------------------------
\begin{comment} 

A \textit{set partition} of $n$ is a partition of $[n]$ into sets. We write a set partition into $k$ parts in the form $T_1 / \cdots / T_k$
where the disjoint union of $T_1,\ldots, T_k$ is $n$.
\begin{definition}
Let $\sigma_1, \ldots, \sigma_d$ be permutations $[n] \mapsto [n]$. We define the \textit{set partition correspondent
to $(\sigma_1, \ldots, \sigma_d)$} to be the partition $T_1/\cdots/T_n$ of $(d + 1)n$ such that 
$T_i = \{i, n + \sigma_1(i), 2n + \sigma_2(i), \ldots, dn + \sigma_d(i)\}.$
\end{definition}
A set partition $T_1 / \cdots / T_l$ of $n$ \textit{contains} a set partition $T_1'/\cdots/T_m'$ of k in the if there exists an injection $f: [k]\mapsto [n]$ and an injection $g: [m']\mapsto [$ such that for each $i\in [m]$,we have $f(T_i')\in $
\begin{definition}
The \textit{permutability} of a set partition $\pi$, which we will call $\pm(\pi)$, is the minimum
$d$ such that there exists a $d$-tuple of permutations $(\sigma_1, \ldots , \sigma_d)$ such that the correspondent  partition $[\sigma_1, \ldots, \sigma_d]$ contains $\pi$. 
\end{definition}

\end{comment}

%----------------------------------End of Comment ------------------------------------

%% file: new_results_outline.tex
\subsection{New Results}\label{new_results_outline}

We prove several new bounds on the extremal functions of graphs and multidimensional matrices using techniques from the extremal theory of matrices, probability, and analysis.
We also develop new methods for bounding the extremal function of hypergraphs in terms of the extremal function of multidimensional matrices.

In section~\ref{graph_extensions}, we prove an equivalence between the asymptotics of the graph extremal function for a class of bipartite graphs and the asymptotics of the matrix extremal function. 
We use the equivalence as well as upper bounds obtained from Cibulka and Kyn\u{c}l \cite{CibKyncl} to prove that $\gex{P}{n} \leq \frac{8}{3}(k+2)^2 2^{4(k+1)}n$ for all permutation graphs $P$ of length $k$.
We use the equivalence to improve the known upper bound for $j$-tuple permutation graphs to $\gex{P}{n} = 2^{O(k)}n$.
The previous bound proven by Weidert \cite{WeidertThesis} was $\gex{P}{n} =2^{O(k\log{k})}n$.
We also generalize the upper bound $2^{O(k^{2/3}(\log k)^{7/3})/(\log \log k)^{1/3}}n$ for the extremal function of almost all permutations matrices \cite{CibKyncl} to the extremal function of almost all permutation graphs.

In section~\ref{hypergraph_extensions}, we generalize the upper bound on graphs in Lemma~\ref{graph_upperbound} to hypergraphs.
For a $d$-permutation hypergraph $P$ of length $k$, we improve the bound $\hyperex{i}{P}{n} = O(n^{d-1})$ obtained by Gunby and P{\'a}lv{\"o}lgyi \cite{GunbyPalv} to $\hyperex{i}{P}{n} = 2^{O(k)}n^{d-1}$.
This also generalizes Geneson and Tian's result \cite{TianGeneson} that $f(Q, d, n) = 2^{O(k)}n^{d-1}$, where $Q$ is a $d$-permutation matrix of length $k$. 
We also sharpen Lemma 7.1 of \cite{GunbyPalv} by bounding the number of hypergraphs avoiding a given $d$-permutation hypergraph to $2^{2^{O(k)}n}$.
Furthermore, our proof extends to when $P$ is the hypergraph associated with a $j$-tuple $d$-permutation matrix of length $k$.

In section~\ref{lowerbounds}, we use the probabilistic method to derive lower bounds for the extremal functions mentioned in this paper.
We generalize a lower bound of a completely filled matrix \cite{TianGeneson} to a lower bound on arbitrary matrices and graphs.
Crowdmath \cite{CrowdMath} proved that for an $r\times c$ binary matrix $B$, if it has more than $r + c - 1$ one entries, then $\ex{B}{n} = \Omega(n\log n)$.
We use the new lower bound to show that if $B$ has more than $r + c - 1$ one entries, then $\ex{B}{n} = \Omega(n^{1 + \epsilon})$ for some $\epsilon > 1$.
We also generalize this lower bound to arbitrary hypergraphs.
Furthermore, we use the lower bounds  for $f(P, d, n)$ for $d$-permutation matrices \cite{TianGeneson} to find lower bounds on $\hyperex{i}{Q}{n}$ where $Q$ is a $d$-permutation hypergraph.
This lower bound shows that our upper bound for the hypergraph extremal function of $d$-permutations is tight up to a constant dependent on $d$.

%% file: graph_extensions.tex
\section{Equivalence of graph and matrix extremal functions}\label{graph_extensions}

In this entire section, unless otherwise stated, let $P$ be a matrix on $[k_1]\times [k_2]$ with a 1-entry at $P_{(k_1, 1)}$.
Let $Q$ be the graph associated with $P$.
\begin{theorem}\label{graph_equivalence_theorem}
$\gex{Q}{n} \sim \ex{P}{n}$.
\end{theorem}

We generalize Corollary 2.2.9 from \cite{MarcThesis}.
\begin{lemma}\label{graph_upperbound}
For all $n\in \mathbb{Z}^+$ $\gex{Q}{n} \leq \ex{P}{n}$.
\end{lemma}

\begin{proof}
Let $A = ([n], E)$ be a graph avoiding $Q$. 
Let $B$ be the $n\times n$ matrix defined by $B_{ij} = 1$ if $\{i, j\}\in E$ and $i < j$, and let $C=([2n], E')$ be the graph associated with $B$.
The number of 1-entries in $B$ is $|E|$.
We also have $\{i, j\}\in E$ if and only if $\{i, j+n\}\in E'$.
Suppose for contradiction that $B$ has more than $\ex{P}{n}$ 1-entries. Then $B$ contains $P$.
Let $P'$ be the submatrix of $B$ that represents $P$, where the rows of $P'$ are $\{r_1, \ldots, r_{k_1}\} \subset [n]$ and the columns of $P'$ are $\{c_1, \ldots, c_{k_2}\} \subset [n]$ with $r_1 < \cdots < r_{k_1}$ and $c_1 < \cdots < c_{k_2}$.
Since the bottom-left 1-entry of $P'$ is a 1-entry in $B$, by construction of $B$, we have $r_{k_1} < c_1$. 
Let $B'$ be the graph associated with $P'$, so $Q$ is contained in $B'$, which is contained in $C$.
Let $G = (V, F)$ be the copy of $Q$ in $B'$.
Then $V = \{r_1, \ldots, r_{k_1}, n + c_1, \ldots, n+ c_{k_2}\}$.
If $f$ is the increasing injection from  $\{r_1, \ldots, r_{k_1}, n + c_1, \ldots, n+ c_{k_2}\}$ to $\{r_1, \ldots, r_{k_1}, c_1, \ldots, c_{k_2}\}$, then $f(G) = Q$, so $G$ is order-isomorphic to $Q$. 
Then since $A$ contains $G$, we have that a contains $Q$, contradiction.
\end{proof}

Also note that we can use a symmetrical argument if $P$ has a 1-entry in the top-right corner.

\begin{lemma}\label{graph_larger_matrix}
For all $n, t\in \mathbb{Z}^+$, we have $\gex{Q}{nt} \geq (t-1)\ex{P}{n}$.
\end{lemma}

\begin{proof}
Let $A$ be a bipartite graph on $[2n]$ with parts $\{1, \ldots, n\}$ and $\{n+1,\ldots, 2n\}$ that avoids $Q$ with $\ex{P}{n}$ edges. 
Let $I, J\subset [n]$ such that the edges of $A$ are $\{i, n + j\}$ for $i\in I$ and $j\in J$. 
Let $G$ be a graph with vertex set $[nt]$ and edges $\{(k-1)n + i, kn + j\}$ for each $i, \in I, j \in J,$ and $k\in [t-1]$. 
We show that $G$ avoids $Q$.
Define intervals $I_k = [(k-1)n + 1, kn]$ for each $k\in [t]$. 
We see that every edge in $G$ connects vertices in consecutive intervals.

For contradiction, suppose $G' = (V', E')$ is subgraph of $G$ isomorphic to $Q$, so $G'$ must also be bipartite.
Let the parts be $V_1$ and $V_2$.
Suppose $G'$ contains vertices from three intervals $I_{x-1}, I_x, I_{x+1}$.
Let $f: [k_1 + k_2] \mapsto V'$ be the isomorphism from $Q$ to $G'$.
Without loss of generality, suppose $f(\{k_1, k_1 + 1\}) = \{v_1, v_2\}$ such that $v_1 \in I_{x-1}$ and $v_2\in I_x$.
Then $V_1 \subset I_{x - 1}$. Since there are no vertices in $I_{x+1}$ adjacent to any vertices in $I_{x - 1}$, it follows that $V_2\subset I_x$. 
If $G'$ contains vertices from only two different intervals $I_x$ and $I_{x+1}$, then $G'$ is order-isomorphic to a subgraph of $A$, so then $G'$ avoids $Q$.
\end{proof}

Now we prove the main theorem of this section.
\begin{sloppypar}
\begin{proof} 
From \cite{fox2013stanley}, we have $\lim_{n\rightarrow\infty}{\frac{\ex{P}{n}}{n}} = c_P$ for some $c_P\in \mathbb{R}$.
Then $\ex{P}{n} = c_P n + o(n)$.
Then from Lemma~\ref{graph_upperbound} and Lemma~\ref{graph_larger_matrix},
$\ex{P}{n^3} \geq\gex{Q}{n^3} \geq (n^2-1)(c_P n + o(n))$, 
so $\gex{Q}{n^3} = n^3c_P  + o(n^3)$, which implies $\lim_{n\rightarrow\infty}{\frac{\gex{Q}{n}}{n}} = c_P$.
\end{proof}
\end{sloppypar}

Lemma~\ref{graph_upperbound} also has some corollaries that improve known bounds in other problems, specifically, when $P$ is a permutation matrix or a $j$-tuple permutation matrix.

Cibulka and Kyn\u{c}l \cite{CibKyncl} proved  that $\ex{P}{n} \leq \frac{8}{3}(k+1)^2 2^{4k}n$ for all $2$-permutation matrices $P$ of length $k$.
Appending a new row and a new column of $P$ to obtain a $(k + 1)\times (k + 1)$ permutation matrix $P'$ with a 1 in the bottom-left corner results in the following corollary:
\begin{corollary}
For all permutation graphs $Q$ of length $k$, we have $\gex{Q}{n} \leq \frac{8}{3}(k+2)^2 2^{4(k+1)}n$.
\end{corollary}

Furthermore, this argument also extends the known bound of the graph extremal function of almost all permutation.
Cibulka and Kyn\u{c}l \cite{CibKyncl} also proved that for almost all $k\times k$ permutation matrices that are $r$-repetition free, we have $\ex{P}{n} = 2^{O(r^{1/3}k^{2/3}(\log{k})^2)}n$. If $P$ is $r$-repetition free, then $P'$ is $(r+1)$-repetition free.

\begin{sloppypar} 
\begin{corollary}
For almost all permutation graphs $Q$ with length $k$, we have $\gex{Q}{n} = 2^{O\left((k^{2/3}(\log k)^{7/3})/(\log \log k)^{1/3}\right)}n.$
\end{corollary}
\end{sloppypar}

A $j$-tuple permutation matrix of length $k$ is a $k\times kj$ matrix that results from replacing each one entry in a permutation matrix with a $1\times j$ matrix of ones and each zero entry with a $1\times j$ matrix of zeros. 
Then the \textit{$j$-tuple permutation graph} is a graph associated with a $j$-tuple permutation matrix.
Geneson and Tian \cite{TianGeneson} proved that $\ex{P}{n} = 2^{O(k)}n$ for all $j$-tuple permutation matrices $P$ of permutations of length $k$.
We improve the bound $\gex{P}{n} \leq 11k^4\binom{2k^2}{2k}n$ from Corollary 3.0.6 of \cite{WeidertThesis} with the same method.

\begin{corollary}
For all $j$-tuple permutation graphs $Q$ of length $k$, we have $\gex{Q}{n} = 2^{O(k)}n$.
\end{corollary}

%% file: hypergraph_extensions.tex
\section{Improved upper bound on hypergraph extremal function}\label{hypergraph_extensions}

We improve the bound found by \cite{GunbyPalv} and provide a more elegant argument by building off of the results of \cite{TianGeneson} and generalizing our Lemma~\ref{graph_upperbound}.

\begin{theorem}\label{main_hypergraph_theorem}
Let $H$ be a fixed $d$-permutation hypergraph with $k$ edges. Then $\hyperex{i}{H}{n} = 2^{O(k)}n^{d-1}$.
\end{theorem}

We find a class of $d$-partite hypergraphs whose extremal functions can be bounded by the extremal functions for their associated $d$-matrices.

\begin{lemma}\label{hypergraph_extremal_lemma}
For some $t, d\in \mathbb{Z}^+$, let $H = ([dt], D)$ be a $d$-uniform $d$-partite hypergraph with parts of size $t$ such that for each $i\in [d-1]$, $H$ has an edge that is a superset of $\{it , it+1\}$. Let $G = ([n], E)$ be a $d$-uniform hypergraph that avoids $H$. If $P$ is the $d$-dimensional matrix associated with $H$, then the number of edges in $G$ is at most $f(P, d, n)$.
\end{lemma}

\begin{proof}
Let $A$ be the $d$-dimensional matrix with side length $n$ such that for each $e = \{k_1,\ldots ,k_d\}\in E$ with $k_1 < \cdots < k_d$, $A$ has a one entry at $a_{k_1, \ldots, k_d} = 1$. Then $A$ has $|E|$ 1-entries.

Suppose for contradiction that $A$ has more than $f(P, d, n)$ 1-entries. 
Then $A$ contains $P$. Let $G' = ([nd], E')$ be the hypergraph associated with $A$, so $G'$ contains $H$. Let $H' = (V, F)$ be the subgraph of $G'$ isomorphic to $H$. 
Then $H'$ is $d$-partite with parts $\{r_1, \ldots, r_{t}\}, \{n + r_{t + 1}, \ldots, n + r_{2t}\}, \ldots, \{(d-1)n + r_{(d-1)t + 1}, \ldots, (d-1)n + r_{dt}\}$.
Assume that for each $i\in [t]$ I have $r_{(i-1)d + 1} < \cdots < r_{id}$.
Let $g: [dt] \mapsto V$ be the isomorphism from $H$ to $H'$. Then $g(\{it, it + 1\}) = \{r_{it}, n + r_{it + 1}\}$ for each $i\in [d-1]$, so the edge in $H'$ that is a superset of $\{r_{it}, n + r_{it + 1}\}$ corresponds to a 1-entry in $A$.

By the construction of $A$ it follows that $r_{it} < r_{it + 1}$. Then $1 \leq r_1 < \cdots < r _{dt} \leq n$.
Then let $H'' = (V', F')$ be the graph defined by $V' = \{r_1, \ldots, r_{dt}\}$ that contains edges $\{r_{i_1}, \ldots, r_{i_d}\}\in F'$ if $A_{(r_{i_1}, \ldots, r_{i_d})} = 1$. Clearly $H''$ is contained in $G$, and $H''$ is order-isomorphic to $H$. Then $G$ contains $H$, contradiction.
\end{proof}

\begin{lemma}\label{hypergraph_k+d-1_lemma}
Let $H$ be a fixed $d$-permutation hypergraph of length $k$. Then there exists a $d$-permutation hypergraph $H'$ of length $k + d - 1$ that contains $H$ and satisfies the conditions of Lemma~\ref{hypergraph_extremal_lemma}.
\end{lemma}

\begin{proof}
Let $P$ be the $d$-matrix of $H$, so $P$ has side length $k$.
Let $Q$ be the $d$-permutation of length $d$ that has one entries at each of the cyclic variants of $(1, \ldots, d)$; i.e. $(1, 2, \ldots, d)$, $(d, 1, \ldots, d-1)$, $\ldots$, $(2, \ldots, d, 1)$.
Construct $P'$ by replacing the one entry at $(1, \ldots, d)$ with $P$, so $P'$ has side length $k + d - 1$.
Let $H'$ be the hypergraph associated with $P'$.
Clearly $H'$ is a $d$-permutation of length $(k + d - 1)$ that contains $H$.
For $i = 2,\ldots, d$, the $(d-i+1)\th$ coordinate of the entry $(i, \ldots, d, 1, \ldots, i-1)$ of $Q$ is $d$.
This entry corresponds to an entry in $P'$ whose $(d - i + 1)\th$ coordinate is $k + d - 1$ and whose $(d - i + 2)\th$ coordinate is $1$.
Then the edge in $H'$ that corresponds to that entry contains $\{(d - i + 1)(k + d - 1) + (k + d - 1), (d - i + 2)(k + d - 1) + 1\}$.
\end{proof}

Now we prove the main theorem of this section.
\begin{proof}
\begin{sloppypar}
If $H$ is a $d$-permutation hypergraph of length $k$, then Lemma~\ref{hypergraph_extremal_lemma}, there exists a $d$-permutation hypergraph $H' = ([(k+d-1)d], D)$ of length $k+d-1$ that contains $H$ and satisfies the conditions of Lemma~\ref{hypergraph_extremal_lemma}. Let $P'$ be the $d$-matrix associated with $H'$.
\end{sloppypar}

Let $G = ([n], E)$ be a hypergraph avoiding $H$, so $G$ also avoids $H'$. Create $G' = ([n], E')$ from $G$ by removing every edge from $E$ with size less than $d$.
Create $G'' = ([n], E'')$ from $G'$ by replacing every edge having more than $kd$ vertices $\{v_1, \ldots, v_l\} \in E'$ with $\{v_1, \ldots, v_{(k+d)d}\}$. 
For each edge in $e\in E''$, there are at most $d$ edges in $E'$ that map to $e$, otherwise $G'$ would contain $H$'.

For each $i = d, d+1, \ldots, kd$, let $G_i = ([n], E_i)$ be the $i$-uniform hypergraph that consists of every edge of $G''$ of size $i$. 

Let $P_{d-1} = P'$, and let $P_i$ be a $d$-permutation matrix of length $k+i$ that contains $P'$ and $P_{i-1}$ such that associated hypergraph of $P_i$ satisfies the conditions of Lemma~\ref{hypergraph_extremal_lemma}.
We construct each $P_i$ by inserting a 1-entry somewhere in $P_{i-1}$ between any consecutive cross sections.

Let $H_i$ be the $d$-permutation hypergraph associated with $P_{i}$.
Since $H_i$ contains $H'$, $G_i$ avoids $H_i$.
From Lemma~\ref{hypergraph_extremal_lemma}, we see that $|E_i| \leq f(P_i, d, n)$.
Since each $P_i$ is contained in $P_{kd}$, it follows that $f(P_i, d, n) \leq f(P_{kd}, d, n)$.

\begin{align*}
|E| &\leq \binom{n}{0} + \cdots + \binom{n}{d-1} + |E'| \\
& \leq dn^{d-1} + d|E''| \\
& \leq dn^{d-1} + d\left(\sum_{i = d}^{kd} |E_i|\right)\\
& \leq dn^{d-1} + d\left(\sum_{i = d}^{kd} f(P_{kd}, d, n)\right)\\
& \leq dn^{d-1} + kd^2\left(2^{O(kd)}(n)^{d-1}\right) \\
& = 2^{O(k)}n^{d-1},
\end{align*}
where the constant hidden in $O(k)$ depends on $d$. 
We used the Theorem 4.1 from \cite{TianGeneson}, which states that $f(P, d, n) = 2^{O(k)}n^{d-1}$ for any $d$-permutation matrix $P$ of length $k$.
Then from \cite{KlazarMarcus}, I have $\hyperex{i}{H}{n} \leq (2kd - 1)(k - 1)\hyperex{e}{H}{n}$, so the result follows.
\end{proof}

Geneson and Tian \cite{TianGeneson} showed that for any $j$-tuple $d$-permutation matrix of length $k$, I have $f(P, d, n) = 2^{O(k)}n^{d-1}$. It is easy to modify the proof of Theorem~\ref{main_hypergraph_theorem} to get the following corollary.
\begin{corollary}
If $P$ is a hypergraph associated with a $j$-tuple $d$-permutation matrix of length $k$, then $\hyperex{i}{P}{n} = 2^{O(k)}n^{d-1}$.
\end{corollary}

%% file: partition_bounds.tex
%\section{Applications to partitions}\label{applications_to_partitions}

%I can use this theorem to bound the number of hypergraphs avoiding a given $d$-permutation hypergraph.

\begin{theorem}\label{number_forbidden_hypergraphs}
Let $H$ be a $d$-permutation hypergraph of length $k$. The number of hypergraphs with vertex set $[n]$ that avoid $H$ is at most $2^{2^{O(k)}n^{d-1}}$.
\end{theorem}
\begin{proof}
Let $M(H, n)$ be the set of hypergraphs on $[n]$ that avoid $H$. 

Let $G\in M(H, tn)$. For each $i\in [n]$, let the intervals $I_i = \{(i-1)t + 1, \ldots, it\}$.
Create a new graph $G'$ on $[n]$ such that for each edge $e$ of $G$, if $e$ has vertices in $I_{k_1}, \ldots, I_{k_l}$, then $G'$ has the edge $\{k_1, \ldots, k_l\}$. Since $G$ contains $G'$, it follows that $G'$ also avoids $H$. Let $f: M(H, tn) \mapsto M(H, n)$ be the map from each $G$ to $G'$. 
For each $G'\in M(H, n)$, at most $(2^t - 1)^{\hyperex{e}{H}{n}}$ graphs $G\in M(tn, H)$ map to $G'$ because each incidence in $G'$ has $2^t - 1$ possible sets of incidences in $G'$ that map to it.
Then
$|M(tn, H)| \leq (2^t - 1)^{\hyperex{i}{H}{n}}|M(n, H)|$.
Let $f(n) = \log{|M(n, H)|}$.
We get $f(tn) = f(n) + 2^{O(k)}\log{(2^t-1)}n^{d-1}$, so iterating this inequality gives 
\[f(t^l) = \sum_{i = 1}^l 2^{O(k)}\log{(2^t-1)}(t^{i-1})^{d-1}.\]
Then $f(t^l) = 2^{O(k)}(t^l)^{d-1}$, so $|M(n, H)| =2^{2^{O(k)}n^{d-1}}$
\end{proof}

%% file: lowerbounds.tex
\section{Lower bounds on extremal functions}\label{lowerbounds}
The following lemma is a generalization of Theorem 2.1 from \cite{TianGeneson}.
\begin{lemma}\label{general_probabilistic_matrix_bound}
Let $B$ be a $d$-matrix on $[k_1]\times \cdots\times [k_d]$ . Then
\[f(B, d, n) = \Omega\left(n^{d - \frac{k_1 + \cdots + k_d - d}{w(B) - 1}}\right).\]
\end{lemma}
\begin{proof}
Let $A$ be a random matrix of side length $n$ such that each entry is a 1 with probability $p$, and each entry is chosen independently of others. 
The expected value of $w(A)$ is $n^d p$. 
We form $A'$ as follows: from every $k_1\times \cdots \times k_d$ submatrix of $A$, if that submatrix represents $B$, then replace a one entry with a 0 so that the submatrix then avoids $B$. 
We see that $A'$ avoids $B$.
Each submatrix represent $B$ with probability $p^{w(B)}$.
The expected number of ones in $A'$ is
\[E[w(A')] = n^dp - p^{w(B)}\left(\prod_{i = 1}^d  \binom{n}{k_i}\right)
\geq n^dp - \frac{(en)^{k_1 + \cdots + k_d}}{{k_1}^{k_1}\cdots {k_d}^{k_d}}p^{w(B)},\]
where we use the Stirling approximation for the inequality.
Choose
$p = \frac{1}{2} n^{-\frac{k_1 + \cdots + k_d - d}{w(B) - 1}}.$
There exists an event such that $w(A') \geq E[w(A')]$, so the result follows. 
\end{proof}

This proves a stronger condition of nonlinearity than the bounds shown by Crowdmath \cite{CrowdMath}.
\begin{corollary}
If $B$ is an $r\times c$ matrix with $w(B) > r+c - 1$, then $\ex{B}{n} = \Omega(n^{1 + \epsilon})$ for some $\epsilon > 0$.
\end{corollary}

We can use the same method to bound the extremal function of all graphs.
\begin{theorem}
Let $G = ([k], E)$ be a graph. Then $\gex{G}{n} = \Omega(n^{2 - \frac{k - 2}{|E| - 1}}).$
\end{theorem}

\begin{sloppypar}
Fox improved the probabilistic lower bound from \cite{TianGeneson}, showing that for almost all $d$-permutation matrices $P$ of length $k$, we have $f(P, d, n) = 2^{\Omega(n^{1/2})}n^{d-1}$.
Klazar and Marcus \cite{KlazarMarcus} observed that if any $d$-matrix $A$ avoids $P$, then the hypergraph associated with $A$ avoids the hypergraph associated with $P$.
We use their observation to obtain the following statement.
\end{sloppypar}

\begin{corollary}
For almost all $d$-permutation hypergraphs $P$ of length $k$, we have
\[\hyperex{i}{P}{n} = 2^{\Omega(k^{1/2})}n^{d-1}.\]
\end{corollary}

Combining our upper and lower bounds shows that for almost all $d$-permutation hypergraphs $P$ of length $k$, we have $\hyperex{i}{P}{n} = 2^{k^{\Theta(1)}}n^{d-1}$, indicating that our bounds are tight.

%% file: conclusion.tex
\section{Conclusion}
We reduced the calculation of the extremal function a class of bipartite graphs to the calculation of their associated matrices by showing an equivalence between the two problems.
We bounded extremal function of $d$-permutation $d$-partite hypergraphs in terms of the extremal function of their associated $d$-matrices.
We also obtained improved lower bounds for the extremal function of all $d$-matrices and graphs with the probabilistic method.

\begin{sloppypar}
One possible future direction for this research would be to show that $f(P, d, n) = 2^{O(k^{2/3 + o(1)})}n^{d-1}$ for almost all $d$-permutation matrices of length $k$.
Using a similar method as the one used by \cite{CibKyncl}, it seems likely that their argument can generalize to $d > 2$.
This would also imply that $\hyperex{i}{Q}{n} = 2^{O(k^{2/3 + o(1)})}n^{d-1}$ for almost all $d$-permutation hypergraphs $Q$ of length $k$.
\end{sloppypar}

Another possible direction would be to apply our results to the extremal function of partitions, studied in \cite{GunbyPalv}.
Gunby and  P{\'a}lv{\"o}lgyi used the hypergraph extremal function to find doubly exponential upper bounds on the number of partitions that avoids a given pattern.
We can use our improved result on the hypergraph extremal function to sharpen the bounds on the partition extremal function.

%% file: acknowledgements.tex
\section{Acknowledgements}
The author would like to thank Dr. Jesse Geneson for his valuable advice and guidance through the duration of this project. The research of the author was also supported by the Department of Mathematics, MIT through PRIMES-USA 2017.